\documentclass[11pt,a4paper]{amsart}
\let\amsmarkboth\markboth    %%%%%%%%Bug tetex 3.0 !!!!!
\usepackage{amsmath,amsfonts,amssymb,amsthm,amscd}
\usepackage[english]{babel}
\usepackage[mathscr]{eucal}
\usepackage{graphicx}
\usepackage{appendix}
\usepackage{color}

\makeatletter
\let\markboth\amsmarkboth
\bbl@redefine#1#2{%
   \def\bbl@arg{#1}%
   \ifx\bbl@arg\@empty
     \toks@{}%
   \else
     \toks@{\noexpand\foreignlanguage{%
              \languagename}{%
              \noexpand\bbl@restore@actives#1}}%
   \fi
   \def\bbl@arg{#2}%
   \ifx\bbl@arg\@empty
     \toks8{}%
   \else
     \toks8{\noexpand\foreignlanguage{%
              \languagename}{%
              \noexpand\bbl@restore@actives#2}}%
   \fi
   \edef\bbl@tempa{\the\toks@}%
   \edef\bbl@tempb{\the\toks8}%
   \protected@edef\bbl@tempa{%
     \noexpand\org@markboth{\bbl@tempa}{\bbl@tempb}}%
   \bbl@tempa
}
\DeclareRobustCommand*\ams@disablelinebreak{\def\\{ \ignorespaces}}
\def\maketitle{\par
   \@topnum\z@ %
   \@setcopyright
   \thispagestyle{firstpage}%
   \uppercasenonmath\shorttitle
   \ifx\@empty\shortauthors \let\shortauthors\shorttitle
   \else \andify\shortauthors
   \fi
   \@maketitle@hook
   \begingroup
   \@maketitle
   \toks@\@xp{\shortauthors}\@temptokena\@xp{\shorttitle}%
   \protected@edef\@tempa{%
     \@nx\markboth{\ams@disablelinebreak
       \@nx\MakeUppercase{\the\toks@}}{\the\@temptokena}}%
   \@tempa
   \endgroup
   \c@footnote\z@
   \@cleartopmattertags
}

\makeatother

\numberwithin{equation}{section}
\newcommand{\om}{\omega}

\newcommand{\dsp}{\displaystyle}
\newcommand{\R}{{\mathbb R}}

\newcommand{\RR}{{\mathbb R}^2}

\newcommand{\curl}{{\rm curl}\,}

\newcommand{\diver}{{\rm div}\,}

\newcommand{\1}{\mathbf 1}

\newtheorem{theorem}{Theorem}[section]
\newtheorem{proposition}[theorem]{Proposition}
\newtheorem{lemma}[theorem]{Lemma}

\theoremstyle{definition}

\theoremstyle{remark}
\newtheorem{remark}[theorem]{Remark}

\newcommand{\dt}{\partial_t}

\begin{document}

%% title page

\title[A uniqueness criterion for the Vlasov-Poisson system]{A uniqueness criterion for unbounded solutions to the Vlasov-Poisson system}

\author[Evelyne Miot]{Evelyne Miot}
\address[E. Miot]{Centre de Math\'ematiques Laurent Schwartz\\ \'Ecole Polytechnique \\ 91128 Palaiseau, France}  \email{evelyne.miot@math.polytechnique.fr}


\date{\today}

\begin{abstract}
We prove uniqueness for the Vlasov-Poisson system in two and three dimensions under the condition that the $L^p$ norms of the macroscopic density growth at most linearly with respect to $p$. This allows for solutions with logarithmic singularities. 
We provide explicit examples of initial data that fulfill the uniqueness condition and that exhibit a logarithmic blow-up. In the gravitational two-dimensional case, such states are intimately related to radially symmetric steady solutions of the system. Our method relies on the Lagrangian formulation for the solutions, exploiting the second-order structure of the corresponding ODEs.

\end{abstract}

\maketitle

\section{Introduction}

The purpose of this article is to establish a uniqueness result for the Vlasov-Poisson system in dimension $n=2$ or $n=3$
\begin{equation}
\label{syst:VP}\begin{cases}
\dsp  \partial_t f+v\cdot \nabla_x f+E\cdot \nabla_v f=0,\quad (t,x,v)\in \R_+\times \R^n\times \R^n \\
\dsp E(t,x)=\gamma \int_{\R^n} \frac{x-y}{|x-y|^n} \rho(t,y)\,dy\\
\dsp \rho(t,x)=\int_{\R^n} f(t,x,v)\,dv,\end{cases}
\end{equation} where $\gamma=\pm 1$.
The system \eqref{syst:VP} is a physical model for the evolution of a system of particles  interacting via a self-induced  force field $E$. The interaction is gravitational if $\gamma=-1$ or Coulombian if $\gamma=1$. The unknown $f=f(t,x,v)\geq 0$ denotes the microscopic density of the particles, and $\rho=\rho(t,x)\geq 0$ their macroscopic density.

\medskip

A wide literature has been devoted to the Cauchy theory for the Vlasov-Poisson system.
 Ukai and Okabe \cite{UO} established global existence and uniqueness of smooth solutions in two dimensions. In any dimension, global existence of weak solutions with finite energy is a result due to Arsenev \cite{A}. In three dimensions,  global existence and uniqueness of compactly supported classical solutions where obtained by Pfaffelmoser \cite{Pf}  by  Lagrangian techniques. Simultaneously,  Lions and Perthame \cite{LP} constructed global weak solutions with finite velocity moments. More precisely, they proved that if  $$f_0\in L^1\cap L^\infty(\R^3)\quad \text{and }\iint_{\R^3\times \R^3} |v|^m f_0<\infty\quad \text{for some }m>3,$$ then there exists a corresponding solution  $f\in L^\infty(\R_+,L^1\cap L^\infty(\R^3))$ such that$$\forall T>0,\quad \sup_{t\in [0,T]}\iint_{\R^3\times \R^3} |v|^m f(t,x,v)\,dx\,dv<\infty.$$If $m>6$ such a solution generates a uniformly bounded force field. 
 We also refer to the works by Gasser, Jabin and Perthame \cite{GJP}, Salort \cite{salort} and Pallard \cite{pallard-1, pallard-2} for further results concerning global existence and propagation of the moments. Another issue in the setting of weak solutions consists in determining sufficient conditions for uniqueness. Robert \cite{robert} established uniqueness among weak solutions that are compactly supported. This result was extended by  Loeper \cite{loeper}, who proved uniqueness in the class of weak solutions with bounded macroscopic density
\begin{equation}
\label{cond:loeper}
\forall T>0,\quad \rho\in L^\infty([0,T],L^\infty(\R^n)).
\end{equation}

The main result of this paper generalizes Loeper's uniqueness condition \eqref{cond:loeper} as follows:
\begin{theorem}\label{thm:VP-1}
Let $T>0$.
There exists at most one weak solution $f\in L^\infty([0,T],L^1\cap L^\infty(\R^n\times \R^n))$ of the Vlasov-Poisson system on $[0,T]$ such that
\begin{equation}\label{condition:miot}
\sup_{[0,T]} \sup_{p\geq 1} \frac{\|\rho(t)\|_{L^p}}{p}<+\infty.
\end{equation}
\end{theorem}

\medskip

Our next task is to determine sufficient conditions on the initial data for which any corresponding weak solution satisfies the uniqueness criterion of Theorem \ref{thm:VP-1}. We observe that \eqref{condition:miot} is fulfilled if for example  
\begin{equation}\label{ineq:log}
\forall t\in [0,T],\quad \rho(t,x)\leq C(1+\ln_-|x-\xi(t)|)
\end{equation} 
for some $\xi(t)\in \R^n$ (see \eqref{eq:stirling}). Such densities where constructed by Caprino, Marchioro, Miot and Pulvirenti \cite{italiens-miot} as solutions of a related equation to \eqref{syst:VP}. On the other hand, there exist solutions of \eqref{syst:VP} that satisfy \eqref{ineq:log} initially, as will be shown in Theorems \ref{thm:VP-3} and \ref{thm:VP-4}. However, in general, it is not clear  whether a logarithmic divergence like \eqref{ineq:log} persists at positive times. In fact, in order to propagate a control on the $L^p$ norms of the macroscopic density we also need a description of the initial data at the microscopic level.
In the above-mentioned previous works \cite{LP,pallard-1,pallard-2,salort}, the condition \eqref{cond:loeper} is met by assuming that the initial data satisfy
\begin{equation*}%\label{condition:uniqueness}
 \forall R>0,\quad \forall T>0, \sup_{t\in [0,T]}\sup_{x\in \R^n}\int_{\R^n} \sup_{|y-x|\leq RT, |v-w|\leq RT^2}f_0(y+vt,w)\,dv<+\infty.
\end{equation*}
In the present paper we shall require instead a suitable control on the velocity moments, having in mind the well-known property that velocity moments control the norms of the density, see \eqref{ineq:lp}:
\begin{theorem} \label{thm:VP-2}
Let $f_0\in L^1\cap L^\infty(\R^n\times \R^n)$ be nonnegative and such that
\begin{equation*}\begin{split}
 \iint_{\R^n\times \R^n} |v|^m f_0(x,v)\,dx\,dv<+\infty
\end{split}\end{equation*}for some $m>n^2-n$.
Let $T>0$ and let $f\in L^\infty([0,T],L^1\cap L^\infty(\R^n\times \R^n))$ be a weak solution provided by  \cite[Theo. 1]{LP}\footnote{The result of \cite{LP} is stated for $n=3$. The case $n=2$ can be obtained by a straightforward adaptation.}.
If $f_0$ satisfies
\begin{equation*}\begin{split}
\forall k\geq 1,\quad \iint_{\R^n\times \R^n} |v|^k f_0(x,v)\,dx\,dv\leq (C_0 k)^{\frac{k}{n}},
\end{split}\end{equation*}for some constant $C_0$, then $f$ satisfies the uniqueness condition \eqref{condition:miot}.
\end{theorem}
Typically, Theorem \ref{thm:VP-2} allows to consider initial densities with compact support in velocity as well as Maxwell-Boltzmann distributions of the type $$f_0(x,v)=e^{-|v|^n}|v|^ph_0(x,v),\quad p\geq 0,\quad h_0\in L^1\cap L^\infty \cap  L^\infty_v(L^1_x).$$
Theorem \ref{thm:VP-2} also does include some initial data with unbounded macroscopic density:
\begin{theorem}\label{thm:VP-3}
There exists a nonnegative $f_0\in L^1\cap L^\infty(\R^n\times \R^n)$ satisfying the assumptions of Theorem  \ref{thm:VP-2} and such that\footnote{Here $\omega_n$ denotes the volume of the unit ball of $\R^n$.} 
\begin{equation*}\begin{split}
&\rho_0(x)=\omega_n \ln_-|x|,\quad \forall x\in \R^n.
\end{split}\end{equation*}
\end{theorem}

\bigskip

Let us next explain the main idea for proving Theorem \ref{thm:VP-1}. The argument of Loeper \cite{loeper} in the context of uniformly bounded macroscopic densities (see also \cite[Theo. 3.1, Chapter 2]{livrejaune}) uses loglipschitz regularity for the force field
\begin{equation*}%\label{ineq:E}
|E(t,x)-E(t,y)|\leq \left(\|\rho(t)\|_{L^1}+\|\rho(t)\|_{L^\infty}\right)|x-y|(1+|\ln |x-y||),
\end{equation*}
which enables to perform a Gronwall estimate involving the distance between the Lagrangian flows associated to the solutions.

The loglipschitz regularity fails in the setting of unbounded densities. However, for $L^p$ solutions, Sobolev embeddings imply that $E$ is H\"older continuous with exponent and semi-norm estimated explicitely in terms of $p$ and $\|\rho(t)\|_{L^p}$, see  Lemma  \ref{lemma:fundamental} below. This estimate turns out to be sufficient to close the Gronwall estimate as $p\to +\infty$ provided the $L^p$ norms satisfy the condition in  Theorem \ref{thm:VP-1}.

\medskip

The Vlasov-Poisson system presents lots of analogies with the Euler equations for two-dimensional incompressible fluids 
\begin{equation}
\label{eq:euler}
\begin{cases}
\dt \om + u\cdot \nabla \om=0\quad \text{on } \R\times \RR,\\
\om=\curl u,\:\diver u=0,
\end{cases}
\end{equation}
where $u:\R \times  \RR
\rightarrow \RR$ is the velocity and $\om=\curl u$ is the vorticity.
Because of their analogous transport structure, both equations \eqref{syst:VP} and \eqref{eq:euler} are often handled similarly, especially for the uniqueness issue. In \cite{loeper}, Loeper extends his uniqueness proof for \eqref{syst:VP} to \eqref{eq:euler}. Also the proof of uniqueness in \cite{robert} applies to both equations. We emphasize that this is not the case in the present paper, as is explained in Remarks \ref{remark:euler} and \ref{remark:euler-2}. This is due to the fact that for the Vlasov-Poisson system the Lagrangian trajectories satisfy a second-order ODE, while for the  Euler equations  they satisfy a first-order ODE. This crucial observation was already exploited in \cite{italiens-miot}, where it was proved that a logarithmic divergence on the macroscopic density still yields enough regularity for the force field to get well-posedness for the corresponding ODE.

\medskip

The paper is organized as follows. In the next Section \ref{sec:1} we recall a H\"older estimate  for a field that controls the force field. As a consequence we derive  a
second-order Gronwall estimate on a distance between the Lagrangian flows of two solutions, which leads to the proof of Theorem \ref{thm:VP-1}.  Section \ref{sec:2} is devoted to the proof of  Theorem \ref{thm:VP-2}. Finally in Section \ref{sec:3} we prove Theorem \ref{thm:VP-3} and we display in Proposition \ref{prop:initial} a large class of initial densities for which uniqueness holds. We conclude by commenting on the link with radially symmetric steady states in the two-dimensional gravitational case.

\medskip

\noindent \textbf{Notation. } In the remainder of the paper, the notation $C$ will denote a constant that can change from one line to another, depending only on $T$, $n$, $\|f\|_{L^\infty([0,T],L^1\cap L^\infty(\R^n\times \R^n))}$, and $\iint |v|^m f_0$ (this latter quantity only for the proof of Theorem \ref{thm:VP-2}) but independent on $p$ and $k$ as $p,k\to +\infty$.

\section{Proof of Theorem \ref{thm:VP-1}}\label{sec:1}

\subsection{Lagrangian formulation for weak solutions}

We consider a weak solution  $f\in L^\infty([0,T],L^1\cap L^\infty(\R^n\times \R^n))$ of \eqref{syst:VP} on $[0,T]$. We assume that $\rho\in L^\infty([0,T],L^1\cap L^p(\R^n))$ for some $p>n$. By potential estimates we have $E=c(n)\nabla \Delta^{-1}\rho\in L^\infty([0,T]\times \R^n)$, and 
\begin{equation}
\label{ineq:bounded}
\|E\|_{L^\infty([0,T],L^\infty)}\leq C_p\|\rho\|_{L^\infty([0,T],L^1\cap L^{p})}.
\end{equation} Moreover, $\nabla E\in L^\infty([0,T],L^p(\R^n))$ by virtue of the Cald\'eron-Zygmund ineqality, see \cite[Theo. 4.12]{javier}. Therefore it follows from DiPerna and Lions theory on transport equations \cite[Theo. III2]{dip-lions} that there exists a map $\Phi=(X,V)\in L_{\text{loc}}^1([0,T]\times \R^n\times\R^n; \R^n\times \R^n)$ such that for a.e. $(x,v)\in \R^n \times \R^n$, $t\mapsto (X,V)(t,x,v)$ is an absolutely continuous integral solution of the ODE
\begin{equation}\label{def:flow}
\begin{cases}
\dsp \dot{X}(t,x,v)=V(t,x,v),\quad X(0,x,v)=x\\
\dsp \dot{V}(t,x,v)=E(t,X(t,x,v)),\quad V(0,x,v)=v.
\end{cases}\end{equation}
%Moreover, $\Phi(t,\cdot,\cdot)$ preserves the Lebesgue measure on $\R^n\times \R^n$. 
Moreover,
\begin{equation}\label{representation}\forall t\in [0,T],\quad f(t)=\Phi(t)_\#f_0\end{equation} which means that $f(t)(B)=f_0\left((\Phi(t,\cdot,\cdot)^{-1}(B)\right)$ for all Borel set $B\subset \R^n$. Such a map is unique and is called Lagrangian flow associated to $E$. We refer also to  \cite[Theo. 5.7]{ambrosio-crippa} for a more recent statement and for further developments on the theory. 

%and
%\begin{equation}
%\label{ineq:holder}
%\sup_{t\in [0,T]} |E(t,x)-E(t,y)|\leq C(p)\: \|\rho\|_{L^\infty(L^1\cap L^{p}}|%x-y|^{1-\frac{n}{p}}.
%\end{equation}
We note that \eqref{ineq:bounded} implies that $t\mapsto\Phi(t,x,v)\in W^{1,\infty}([0,T])$ for a.e. $(x,v)\in  \R^n\times\R^n$.

\medskip

As a byproduct of our analysis we shall see in Paragraph \ref{subsec} that under the assumptions of Theorem \ref{thm:VP-1} the Lagrangian flow actually corresponds to the classical notion of  flow. 
\subsection{Estimate on the Lagrangian trajectories}
 We consider two solutions $f_1$ and $f_2\in L^\infty([0,T],L^1\cap L^\infty(\R^n\times\R^n))$  such that $\rho_1$ and $\rho_2$ belong to $L^\infty([0,T],L^p(\R^n))$ for some $p>n$.
 Denoting by $\Phi_1=(X_1,V_1)$ and $\Phi_2=(X_2,V_2)$ the corresponding Lagrangian flows, we introduce the distance
\begin{equation*}%\label{def:distance}
\mathcal{D}(t)=\iint_{\R^n\times \R^n} |X_1(t,x,v)-X_2(t,x,v)|f_0(x,v)\,dx\,dv.\end{equation*} 

\medskip

We infer from \eqref{def:flow} that
\begin{equation}\label{ineq:X}
\begin{split}
|X_1(t,x,v)-X_2(t,x,v)|\leq \int_0^t \int_0^s |E_1(\tau,X_1(\tau,x,v))-E_2(\tau,X_2(\tau,x,v))|\,d\tau\,ds.
\end{split}
\end{equation}
In particular, $\sup_{(x,v)}|X_1(t,x,v)-X_2(t,x,v)|\leq CT^2 (\|E_1\|_{L^\infty}+\|E_2\|_{L^\infty}),$
which shows that $\mathcal{D}$ defines a continuous function on $[0,T]$.
The purpose of this paragraph is to establish the estimate
\begin{proposition}\label{prop:gronwall}
For all $t\in [0,T]$ and for all $p>n$,
$$\mathcal{D}(t)\leq C\, p \max\left(1+\|\rho_1\|_{L^\infty([0,T],L^p)},\|\rho_2\|_{L^\infty([0,T],L^p)}\right)\int_0^t \int_0^s \mathcal{D}(\tau)\,d\tau\,ds.$$
\end{proposition}

\medskip 
The proof of Proposition \ref{prop:gronwall} relies on the following potential estimate, the proof of which is postponed at the end of this paragraph.
\begin{lemma}\label{lemma:fundamental}There exists $C>0$ such that for all $p>n$ and $g\in L^1\cap L^p(\R^n)$,
\begin{equation*}
\int_{\R^n} \left| \frac{x-z}{|x-z|^n}-\frac{y-z}{|y-z|^n}\right| |g(z)|\,dz\leq C p (\|g\|_{L^p}+\|g\|_{L^1}) |x-y|^{1-\frac{n}{p}}.
\end{equation*}

\end{lemma}

\begin{remark}\label{remark:regularity}Setting $E[g]=x/|x|^n \ast g=c(n)\nabla \Delta^{-1}g$ we observe that Lemma \ref{lemma:fundamental} implies the estimate \begin{equation}\label{ineq:Eg}|E[g](x)-E[g](y)|\leq Cp(\|g\|_{L^p}+\|g\|_{L^1})|x-y|^{1-n/p}.\end{equation} This latter inequality can be obtained by combining
Morrey's inequality, which implies that $|E[g](x)-E[g](y)|\leq C\|\nabla E[g]\|_{L^p}|x-y|^{1-\frac{n}{p}}$, and  Calder\'on-Zygmund inequality, see \cite[Theo. 4.12]{javier}, which  implies that  $\|\nabla E[g]\|_{L^p}  \leq C p \|g\|_{L^p}$.
\end{remark}

\noindent \textbf{Proof of Proposition \ref{prop:gronwall}.}

\medskip

By \eqref{ineq:X}, we have
\begin{equation*}
\begin{split}
\mathcal{D}(t)&\leq \int_0^t \int_0^s
 \left(\iint_{\R^n\times \R^n} |E_1(\tau,X_1(\tau,x,v))-E_2(\tau,X_2(\tau,x,v))|f_0(x,v)\,dx\,dv\right)\,d\tau\,ds\\
 &\leq \int_0^t \int_0^s
 \left(\iint_{\R^n\times \R^n} |E_1(\tau,X_1(\tau,x,v))-E_1(\tau,X_2(\tau,x,v))|f_0(x,v)\,dx\,dv\right)\,d\tau\,ds\\
 &+ \int_0^t \int_0^s
 \left(\iint_{\R^n\times \R^n} |E_1(\tau,X_2(\tau,x,v))-E_2(\tau,X_2(\tau,x,v))|f_0(x,v)\,dx\,dv\right)\,d\tau\,ds\\
 &\leq I+J.
 \end{split}\end{equation*}

First, applying \eqref{ineq:Eg} to $E_1$ and using that $\rho_1\in L^\infty([0,T],L^1(\R^n))$ we obtain
\begin{equation*}\begin{split}
\iint_{\R^n\times \R^n}& |E_1(\tau,X_1(\tau,x,v))-E_1(\tau,X_2(\tau,x,v))|f_0(x,v)\,dx\,dv\\&\leq C p \left(1+\|\rho_1\|_{L^\infty(L^p)}\right)\iint_{\R^n\times \R^n}
|X_1(\tau,x,v)-X_2(\tau,x,v)|^{1-\frac{n}{p}}f_0(x,v)\,dx\,dv.\end{split}\end{equation*}
Therefore by Jensen's inequality we find
\begin{equation}\label{ineq:I}
I\leq Cp \left(1+\|\rho_1\|_{L^\infty(L^p)}\right) \int_0^t \int_0^s \mathcal{D}(\tau)^{1-\frac{n}{p}}\,d\tau\,ds.
\end{equation}

Next, inserting that $f_2(\tau)=\Phi_2(\tau)_\#f_0$, we obtain
\begin{equation*}
 \begin{split}
\iint_{\R^n\times \R^n}& |E_1(\tau,X_2(\tau,x,v))-E_2(\tau,X_2(\tau,x,v))|f_0(x,v)\,dx\,dv\\&=\iint_{\R^n\times \R^n} |E_1(\tau,x)-E_2(\tau,x)|f_2(\tau,x,v)\,dx\,dv.\end{split}\end{equation*}
On the other hand, since $f_1(\tau)=\Phi_1(\tau)_\#f_0$ and $f_2(\tau)=\Phi_2(\tau)_\#f_0$,
\begin{equation}\label{eq:E}
\begin{split}
E_1(\tau,x)-E_2(\tau,x)
&=\gamma\iint_{\R^n\times \R^n} 
\left(\frac{x-X_1(\tau,y,w)}{|x-X_1(\tau,y,w)|^n}-
\frac{x-X_2(\tau,y,w)}{|x-X_2(\tau,y,w)|^n}\right)f_0(y,w)\,dy\,dw.\end{split}
 \end{equation}
 Therefore, we obtain by Fubini's theorem
 \begin{equation*}\begin{split}
&\iint_{\R^n\times \R^n} |E_1(\tau,X_2(\tau,x,v))-E_2(\tau,X_2(\tau,x,v))|f_0(x,v)\,dx\,dv\\&\leq 
\int_{\R^n}\left| \iint_{\R^n\times \R^n} 
\left(\frac{x-X_1(\tau,y,w)}{|x-X_1(\tau,y,w)|^n}-
\frac{x-X_2(\tau,y,w)}{|x-X_2(\tau,y,w)|^n}\right)f_0(y,w)\,dy\,dw\right|\rho_2(\tau,x)\,dx\\
&\leq \iint_{\R^n\times \R^n}\left( \int_{\R^n} 
\left|\frac{x-X_1(\tau,y,w)}{|x-X_1(\tau,y,w)|^n}-
\frac{x-X_2(\tau,y,w)}{|x-X_2(\tau,y,w)|^n}\right|\rho_2(\tau,x)\,dx\right)f_0(y,w)\,dy\,dw\\
&\leq \iint_{\R^n\times \R^n} C p\left(\|\rho_2(\tau)\|_{L^1}+ \|\rho_2(\tau)\|_{L^p}\right) |X_1(\tau,y,w)-X_2(\tau,y,w)|^{1-\frac{n}{p}}f_0(y,w)\,dy\,dw,\end{split}\end{equation*}
where we have applied Lemma \ref{lemma:fundamental} in the last inequality.
Hence Jensen's inequality yields
\begin{equation}\label{ineq:J}
J\leq C\, p\left(1+\|\rho_2\|_{L^\infty(L^p)}\right)  \int_0^t \int_0^s \mathcal{D}(\tau)^{1-\frac{n}{p}}\,ds\,d\tau.
\end{equation}

The conclusion follows from \eqref{ineq:I} and \eqref{ineq:J}.

\begin{remark}\label{remark:euler}A similar function can be introduced to establish uniqueness for \eqref{eq:euler} with bounded vorticity, see e.g.  \cite[Theo. 3.1, Chapter 2]{livrejaune},
$$\tilde{\mathcal{D}}(t)=\int_{\R^2}|X_1(t,x)-X_2(t,x)||\omega_0(x)|\,dx,$$
where $X_1$ and $X_2$ denote the Lagrangian flows
\begin{equation*}
\label{def:flow-euler}
\dot{X}_i(t,x)=u_i(t,X_i(t,x)),\quad X(0,x)=x.\end{equation*}
By the same arguments as in the proof of Proposition \ref{prop:gronwall}, it satisfies \begin{equation*}\begin{split}
\tilde{\mathcal{D}}(t)&\leq C\, p \max\left(\|\omega_1\|_{L^\infty(L^1\cap L^p)},\|\omega_2\|_{L^\infty(L^1\cap L^p)}\right)\int_0^t  \tilde{\mathcal{D}}^{1-\frac{2}{p}}(s)\,ds\end{split}\end{equation*}
therefore, by conservation of the $L^p$ norms of the vorticity,
\begin{equation*}\begin{split}
\tilde{\mathcal{D}}(t)&\leq C\, p \|\omega_0\|_{L^1\cap L^p}\int_0^t  \tilde{\mathcal{D}}^{1-\frac{2}{p}}(s)\,ds.\end{split}\end{equation*} 
\end{remark}

\medskip
\medskip

\noindent \textbf{Proof of Lemma \ref{lemma:fundamental}.}

The proof for $p=\infty$ is well-known, see e.g. \cite[Chapter 8]{majda-bertozzi} for the case $n=2$. When $p<+\infty$ it is obtained by very similar arguments, but we provide the full details because we are not aware of any reference in the literature. Let $p_0>n$. By H\"older inequality, we have 
\begin{equation*}\begin{split}
\sup_{x\in \R^n}
\int_{\R^n} \left| \frac{x-z}{|x-z|^n}\right| |g(z)|\,dz&\leq
\sup_{x\in \R^n}
\left(\int_{|x-z|\leq 1}  \frac{|g(z)|}{|x-z|^{n-1}} \,dz+
\int_{|x-z|\geq 1}  \frac{|g(z)|}{|x-z|^{n-1}} \,dz\right)\\
&\leq \|g\|_{L^{p_0}}\||z|^{-n+1}\|_{L^{p_0'}(B(0,1))}+ \|g\|_{L^1}\\
&\leq C_{p_0}( \|g\|_{L^1}+ \|g\|_{L^{p_0}}),\end{split}
\end{equation*}
with $C_{p_0}$ depending only on $p_0$. Hence it suffices to establish Lemma \ref{lemma:fundamental} for $|x-y|<1$. Let us introduce $d=|x-y|$ and $A=(x+y)/2$. We split the integral as
 \begin{equation*}\begin{split}
&\int_{\R^n}  \left| \frac{x-z}{|x-z|^n}-\frac{y-z}{|y-z|^n}\right| |g(z)|\,dz=
\int_{\R^n\setminus B(A,1)} \left| \frac{x-z}{|x-z|^n}-\frac{y-z}{|y-z|^n}\right| |g(z)|\,dz\\
&+\int_{B(A,1)\setminus B(A,d)} \left| \frac{x-z}{|x-z|^n}-\frac{y-z}{|y-z|^n}\right| |g(z)|\,dz+
\int_{B(A,d)}  \left| \frac{x-z}{|x-z|^n}-\frac{y-z}{|y-z|^n}\right| |g(z)|\,dz\\
&=I+J+K.
\end{split}
\end{equation*}

For $|z-A|\geq 1$ we have $\min (|x-z|,|y-z|)\geq 1-d/2\geq 1/2$, hence
 \begin{equation*}\begin{split}
I&\leq \int_{\R^n\setminus B(A,1)}\frac{|g(z)|}{|x-z|^{n-1}}\,dz+\int_{\R^n\setminus B(A,1)}\frac{|g(z)|}{|y-z|^{n-1}}\,dz \leq C\|g\|_{L^1}.
\end{split}\end{equation*}

Next, applying first H\"older inequality, then the mean-value theorem, we obtain
\begin{equation*}\begin{split}
J&\leq\|g\|_{L^p}\left( \int_{B(A,1)\setminus B(A,d)}  \left| \frac{x-z}{|x-z|^n}-\frac{y-z}{|y-z|^n}\right|^{p'}\,dz\right)^{1/p'}\\
& \leq \|g\|_{L^p}\:d \left(\int_{B(A,1)\setminus B(A,d)} \sup_{u\in [x,y]} \frac{1}{|u-z|^{np'}} \,dz\right)^{1/p'}.
\end{split}
\end{equation*}
Now, for $ |z-A|\geq d$ we have $|u-z|\geq |z-A|-|u-A|\geq |z-A|/2$ for any $u\in[x,y]$. Therefore 
\begin{equation*}\begin{split}
J&\leq C d  \|g\|_{L^p} \left(\int_{B(A,1)\setminus B(A,d)}  \frac{1}{|z-A|^{np'}} \,dz\right)^{1/p'}
%&\leq 2^n d  \|f\|_{L^p} \left(\int_d^1r^{n-1-np'}  \,dr\right)^{1/p'}\\
\leq C d\, \|g\|_{L^p} d^{n(1-\frac{1}{p'})}(p'-1)^{-\frac{1}{p'}}\end{split}
\end{equation*}
hence
\begin{equation*}\begin{split}
J&\leq C p\, \|g\|_{L^p}  d^{1-\frac{n}{p}}.
\end{split}
\end{equation*}

Applying again H\"older inequality, we obtain
\begin{equation*}
\begin{split}
K&\leq \|g\|_{L^p} \left(\int_{B(A,d)}  \frac{1}{|x-z|^{p'(n-1)}}\,dz\right)^{1/p'}+ \|g\|_{L^p} \left(\int_{B(A,d)}  \frac{1}{|y-z|^{p'(n-1)}}\,dz\right)^{1/p'}.
\end{split}
\end{equation*}
Since for $|z-A|\leq d$ we have $\max(|x-z|,|y-z|)\leq 3d/2$, we finally obtain
\begin{equation*}\begin{split}
K\leq 2 \|g\|_{L^p}  \left(\int_{B(0,3d/2)} \frac{1}{|u|^{p'(n-1)}}\right)^{1/p'}\leq C\|g\|_{L^p} \: d^{1-\frac{n}{p}}.
\end{split}
\end{equation*}

\medskip

\subsection{Proof of Theorem \ref{thm:VP-1} }

Given two solutions $f_1$ and $f_2$ of \eqref{syst:VP} satisfying the assumptions of Theorem \ref{thm:VP-1}, let $\mathcal{D}$ be the corresponding distance function. Since $\max(\|\rho_1\|_{L^\infty(L^p)},\|\rho_2\|_{L^\infty(L^p)})\leq C p$ by assumption, 
Proposition \ref{prop:gronwall} implies that
$$\mathcal{D}(t)\leq C\, p^2 \int_0^t \int_0^s \mathcal{D}^{1-\frac{n}{p}}(\tau)\,d\tau\,ds.$$
Let $\mathcal{F}(t)=\int_0^t \int_0^s \mathcal{D}^{1-\frac{n}{p}}(\tau)\,d\tau\,ds.$ Since $\mathcal{D}\in C([0,T])$ we have $\mathcal{F}\in C^2([0,T])$, with
$$ \forall t\in [0,T],\quad \mathcal{F}''(t)\leq C \,p^{2} \mathcal{F}^{1-\frac{n}{p}}(t).$$
We next argue similarly as in the proof of Lemma 4 in \cite{italiens-miot}. We multiply the previous inequality by $\mathcal{F}'(t)\geq 0$ and integrate on $[0,t]$. We obtain
$$\forall t\in [0,T],\quad(\mathcal{F}'(t))^2\leq C\,p^2 \mathcal{F}(t)^{2-\frac{n}{p}}$$
therefore
$$\forall t\in [0,T],\quad \mathcal{F}'(t)\leq C \, p \mathcal{F}(t)^{1-\frac{n}{2p}}.$$
We now conclude as in the proof of the uniqueness of bounded solutions of the 2D Euler equations, see e.g. \cite{yudovich-1, majda-bertozzi}: integrating the above inequality yields
$$\forall p>n,\quad \forall t\in [0,T],\quad \mathcal{F}(t)\leq (Ct)^{\frac{2p}{n}}.$$
Letting $p\to +\infty$ we obtain that $\mathcal{F}(t)=0$ for $t\in [0,1/C]$. Repeating the argument of intervals of length $1/C$ we finally prove that $ \mathcal{F}$, therefore also $\mathcal{D}$, vanishes on $[0,T]$. This implies that for all $t\in [0,T]$ we have $X_1(t,\cdot,\cdot)=X_2(t,\cdot,\cdot)$  $f_0\,dx\,dv$ - a.e. We infer from \eqref{eq:E} that for all $t\in [0,T]$, $E_1(t,\cdot)=E_2(t,\cdot)$ on $\R^n$. By \eqref{def:flow}, it follows that $V_1(t,\cdot,\cdot)=V_2(t,\cdot,\cdot)$ on $\R^n\times \R^n$. We conclude that for all $t\in [0,T]$ we have $f_1(t,\cdot,\cdot)=f_2(t,\cdot,\cdot)$  a.e. on $\R^n\times \R^n$.

\begin{remark}
\label{remark:euler-2} In the setting of \eqref{eq:euler}, the estimate obtained for $\tilde{\mathcal{D}}$ in Remark \ref{remark:euler} yields $$ \forall p>2,\quad \tilde{\mathcal{D}}(t)\leq (C \|\omega_0\|_{L^p}\, t)^p,$$ which does not enable to conclude that $\mathcal{D}=0$ as above unless $\omega_0\in L^\infty.$   
\end{remark}

\subsection{The Lagrangian flow is the classical flow}\label{subsec}
We conclude this section with the following remark: let $f$ be a weak solution of \eqref{syst:VP} satisfying the assumptions of Theorem \ref{thm:VP-1}. In view of Remark \ref{remark:regularity} we have
$$ \forall p>n, \sup_{t\in [0,T]}|E(t,x)-E(t,y)|\leq Cp^2|x-y|^{1-\frac{n}{p}}.$$
By space continuity of $E$, Ascoli-Arzela's theorem implies that {for all} $(x,v)\in \R^n\times \R^n$ there exists a curve $\gamma\in W^{1,\infty}([0,T];\R^n\times\R^n)$ which is a solution to the ODE \eqref{def:flow}.  Moreover, if $\gamma_1$ and $\gamma_2$ are two such integral curves then $d(t)=\int_0^t\int_0^s|\gamma_1-\gamma_2|(\tau)\,d\tau\,ds$ satisfies $d''\leq Cp^2d^{1-n/p}$. So by exactly the same arguments as in the proof of Theorem \ref{thm:VP-1} above, $d=0$ on $[0,T]$. This means that the ODE \eqref{def:flow} is well-posed for all $(x,v)\in \R^n\times \R^n$ and that the Lagrangian flow actually is a classical flow.

\section{Proof of Theorem \ref{thm:VP-2}}\label{sec:2}

We start by recallig an elementary inequality, which can be found in  \cite[(14)]{LP} for the case $n=3$, and which can be easily adapted to the case $n=2$. Let $f\in L^1\cap L^\infty(\R^n\times \R^n)$ be nonnegative and $\rho_f(x)=\int f(x,v)\,dv$. Then
\begin{equation}
\label{ineq:lp}
\begin{split}
\forall k\geq 1,\quad \|\rho_f\|_{L^{\frac{k+n}{n}}(\R^n)}\leq C\|f\|_{L^\infty}^{\frac{k}{k+n}}\left(\iint_{\R^n\times \R^n} |v|^k f(x,v)\,dx\,dv\right)^{\frac{n}{k+n}},
\end{split}\end{equation}
where $C$ is a constant independant on $k$.

\medskip

Now, let $f_0$ satisfy the assumptions of Theorem \ref{thm:VP-2} and let $f$ be any weak solution on $[0,T]$ with this initial data given by \cite[Theo. 1]{LP}. By construction we have 
\begin{equation}\label{moment:m}
\sup_{t\in[0,T]}\iint_{\R^n\times \R^n} |v|^m f(t,x,v)\,dx\,dv<+\infty.
\end{equation} In view of \eqref{ineq:lp}, in order to control the norms $\|\rho(t)\|_{L^p}$ for large $p$ it suffices to prove that 
\begin{equation*}
\label{cond:3}
\forall k>0, \quad \sup_{t\in [0,T]}\|f(t)\|_{L^\infty}^{\frac{k}{k+n}}M_k(t)^{\frac{n}{k+n}}\leq C k,
\end{equation*}
where
 $$M_k(t)=\iint_{\R^n\times \R^n} |v|^k f(t,x,v)\,dx\,dv=\iint_{\R^n\times \R^n} |V(t,x,v)|^k f_0(x,v)\,dx\,dv.$$  Since $f\in L^\infty([0,T],L^\infty(\R^n))$ this amounts to showing that
\begin{equation}
\label{cond:3}
\forall k>0, \quad \sup_{t\in [0,T]}M_k(t)^{\frac{n}{k+n}}\leq C k.
\end{equation}
At this stage it is not known whether all the $M_k(t)$ remain finite for $t>0$. We prove next that this is indeed the case and that \eqref{cond:3} can be achieved thanks to \eqref{moment:m} in a much easier way as for the propagation \eqref{moment:m} itself, which is the heart of the matter of \cite{LP}. As a matter of fact, since $m>n^2-n$ we infer from \eqref{ineq:lp} and \eqref{moment:m} that $\rho\in L^\infty([0,T],L^{p_0}(\R^n))$ with $p_0=(m+n)/n>n$ depending only on $n$ and $m$. It follows that $E\in L^\infty([0,T],L^\infty(\R^n))$ by \eqref{ineq:bounded}.
%with$$\|E\|_{L^\infty([0,T],L^\infty(\R^n))}\leq C(p_0)\|\rho\|_{L^\infty([0,T],L^1\cap L^{p_0}(\R^n)},$$for any $p_0>n$.

For $k>m$, we have by \eqref{def:flow} 
 \begin{equation*}\begin{split}
 |V(t,x,v)|^k&\leq |v|^k+k\int_0^t |V(s,x,v)|^{k-1}|E(s,X(s,x,v)|\,ds\\
 &\leq |v|^k +k\|E\|_{L^\infty([0,T]\times \R^n)}\int_0^t |V(s,x,v)|^{k-1}\,ds.
 \end{split}\end{equation*}
 Integrating with respect to $f_0(x,v)\,dx\,dv$ we get
 \begin{equation*}
 \begin{split}
 M_k(t)&\leq M_k(0)+k\|E\|_{L^\infty([0,T]\times \R^n)}\int_0^t M_{k-1}(s)\,ds.
 \end{split}
 \end{equation*} By induction, we first infer that $\sup_{t\in [0,T]}M_k(t)$ is finite for any $k>m$. On the other hand,
we obtain by H\"older inequality
 $$M_{k-1}(s)\leq \|f(s)\|_{L^1}^{\frac{1}{k}}M_k(s)^{1-\frac{1}{k}},$$
 therefore, since $\|f(s)\|_{L^1}=\|f_0\|_{L^1}$ by \eqref{representation} we get
 \begin{equation*}\begin{split}
M_k(t)&\leq M_k(0)+Ck \int_0^t M_k(s)^{1-\frac{1}{k}}\,ds.
 \end{split}
 \end{equation*}
Integrating this Gronwall inequality leads to
 \begin{equation*}
\sup_{t\in [0,T]} M_k(t)^{\frac{1}{k}}\leq M_k(0)^{\frac{1}{k}}+C.
\end{equation*}
By assumption on $M_k(0)$ we find
 \begin{equation*}
\sup_{t\in [0,T]} M_k(t)^{\frac{1}{k}}\leq (C_0k)^\frac{1}{n}+C\leq (Ck)^\frac{1}{n}
\end{equation*}
therefore, finally,
 \begin{equation*}
\sup_{t\in [0,T]} M_k(t)^{\frac{n}{n+k}}\leq Ck,
\end{equation*}
and the conclusion follows.

\section{Proof of Theorem \ref{thm:VP-3}} \label{sec:3}

\subsection{Seeking for initial data.}

In this section we construct a collection of initial densities that satisfy the assumptions of Theorem \ref{thm:VP-2} and that do not necessarily enter in the framework of Loeper's uniqueness condition. We will consider nonnegative measurable functions $\varphi$ on $\R$ such that \begin{equation}\label{compact}\text{supp}(\varphi)\subset ]-\infty, M]\text{ for some } M\in \R.\end{equation}

\begin{proposition}\label{prop:initial}
Let $\varphi\in L^\infty(\R,\R_+)$ satisfy \eqref{compact}. Let  $\Phi:\R^n\to \R$ and $a:\R^n\times\R^n\to \R_+$ be two measurable functions. We set
\begin{equation}\label{def:initial}
f_0(x,v)=\varphi\left(|v|^2+\Phi(x)+a(x,v)\right).\end{equation}
We assume that $\rho_0=\int f_0\,dv$ has compact support in $B\subset \R^n$, and that
\begin{equation*}
\label{cond:phi}
\forall p\geq 1,\quad \int_{B}\left(M-\Phi(x)\right)_+^{p}\,dx\leq (C_0 p)^{\frac{2p}{n}},
\end{equation*}for some constant $C_0$.
Then any initial density given by
$$f_0\,h_0,\quad \text{where } h_0\in  L^\infty(\R^n\times \R^n),$$
satisfies the assumptions of Theorem \ref{thm:VP-2}.
\end{proposition}

\begin{proof}

Since for $(x,v)\in  \text{supp}f_0$ we have $|v|^2\leq M-\Phi(x)$, we obtain
\begin{equation*}\begin{split}
\iint_{\R^n\times \R^n}& |v|^kf_0(x,v)h_0(x,v)\,dx\,dv\\
&\leq \omega_n \|h_0\|_{L^\infty}
 \int_{B}\left(M-\Phi(x)\right)_+^{\frac{k}{2}}\rho_0(x)\,dx\\
&\leq \omega_n \|h_0\|_{L^\infty}\|\varphi\|_{L^\infty} \int_{B}\left(M-\Phi(x)\right)_+^{\frac{k+n}{2}}\,dx.\end{split}\end{equation*} Finally, the condition of Theorem \ref{thm:VP-2} is fulfilled provided
\begin{equation*}
\forall k\geq 1,\quad 
\int_{B}\left(M-\Phi(x)\right)_+^{\frac{k+n}{2}}\,dx\leq (C_0 k)^{\frac{k}{n}},
\end{equation*}
and this concludes the proof.
\end{proof}

%\begin{remark}
%In Proposition \ref{prop:initial} one can replace the assumption on the support of $\rho_0$ by the requirement that $\varphi$ is nonincreasing, in which case the assumption on the $L^p$ norms is replaced by
%\begin{equation*}
%\label{cond:phi}
%\forall p\geq 1,\quad \int_{\R^n}\varphi(\Phi(x))\left(M-\Phi(x)\right)_+^{p}\,dx\leq (C_0 p)^{\frac{2p}{n}},
%\end{equation*}for some constant $C_0$.
%\end{remark}

\subsection{Proof of Theorem \ref{thm:VP-3}.}

 We consider an initial density given by \eqref{def:initial} with the choice
$$\varphi=\1_{\R_-},\quad \Phi(x)=-(\ln_- |x|)^{\frac{2}{n}},\quad a=0,$$
so that
\begin{equation*}
\rho_0(x)=\left|\left\{v\::|v|^2-(\ln- |x|)^{\frac{2}{n}}\leq 0\right\}\right|=\omega_n \ln_-|x|,\quad \forall x\in \R^n.
\end{equation*}
Besides, a straightforward computation yields
\begin{equation}\label{eq:stirling}
\forall p\geq 1,\quad \int_{\R^n}(\ln_- |x|)^{p}\,dx=\sigma_n n^{-(p+1)}p!
\end{equation}where $\sigma_n$ denotes the surface of $\partial B(0,1)$, so that by Stirling's formula we get
\begin{equation}\label{ineq:verif}
\forall p\geq 1,\quad \int_{\R^n}(\ln_- |x|)^{\frac{2p}{n}}\,dx \leq (C p)^{\frac{2p}{n}}.
\end{equation}
The conclusion follows by invoking Proposition \ref{prop:initial}.

\subsection{Steady states in the two-dimensional gravitational case}

In this last paragraph we focus on  the Vlasov-Poisson equation  \eqref{syst:VP} in the gravitational case for $n=2$, which can be rewritten as 
\begin{equation}
\label{syst:VP-bis}\begin{cases}
\dsp  \partial_t f+v\cdot \nabla_x f-\nabla U\cdot \nabla_v f=0 \quad \text{on }\R_+\times \R^2\\
\dsp U(t,x)=\int_{\R^2}\ln|x-y|\rho(t,y)\,dy \\
\dsp \rho(t,x)=\int_{\R^n} f(t,x,v)\,dv.\end{cases}
\end{equation}
Every function of the form
\begin{equation}\label{def:steady}
\overline{f}(x,v)=\varphi\left(\frac{|v|^2}{2}+U(x)\right),
\end{equation}
with $\varphi\in C^1(\R,\R)$, is a stationary solution of \eqref{syst:VP-bis}. Existence of steady states of the form \eqref{def:steady} and their stability properties, especially in three dimensions, have been studied intensively (see \cite{batt-morrison-rein, dolbeault, guo-rein, lmr} and references therein).  By variational methods, Dolbeault, Fern\'andez and S\'anchez \cite[Theo. 1, Theo. 22]{dolbeault} obtained the existence of a steady solution $\overline{f}\in L^1(\R^2\times \R^2)$ of the form \eqref{def:steady}, where  $\varphi$ is continuous, nonincreasing and satisfies \eqref{compact}. Moreover $\overline{\rho}=\int \overline{f}\,dv$ is radially symmetric, compactly supported in $B(0,1)$, and $U$ is of class $C^1$ on $\R^2\setminus \{0\}$. In particular, $U$ has the simple expression, see \cite[Lemma 12]{dolbeault}
\begin{equation*}\begin{split}
U(x)&=\ln |x|\int_{|y|\leq |x|} \overline{\rho}(|y|)\,dy+
\int_{|y|>|x|}\ln {|y|}\rho(|y|)\,dy\\
&=\ln |x|\left(\int_{\R^2} \overline{\rho}(|y|)\,dy\right)+
\int_{|y|>|x|}\ln\left(\frac{|y|}{|x|}\right)\rho(|y|)\,dy.\end{split}
\end{equation*}
Note that $U$ is well defined for all $x\neq 0$ in view of the assumption on the support of $\overline{\rho}$. We remark that $\overline{f}$ does not have to belong to $L^\infty(\R^2\times \R^2)$.

\begin{theorem}\label{thm:VP-4}
Let $\overline{f}$ be as above. Then for any $K>0$, any initial density given by
$$\overline{f}\,\1_{ \{\overline{f}\leq K\}}\,h_0,\quad \text{where } h_0\in L^\infty(\R^2\times \R^2),$$
satisfies the assumptions of Theorem \ref{thm:VP-2}.
\end{theorem}

\begin{proof}Note that $\overline{f}\,\1_{ \{\overline{f}\leq K\}}\,h_0\in L^1\cap L^\infty(\R^2\times \R^2)$.
Since $\overline{\rho}$ is supported in $B(0,1)$ we have $U(x)=\ln|x| (\int \overline{\rho})$ for $|x|\geq 1$, from which we infer that $\overline{f}(x,v)=0$ whenever $|x|\geq N=\exp(M/(\int \overline{\rho}))$. 
In addition, we observe that $\overline{f}$ takes the form \eqref{def:initial}, where we have set \begin{equation*}
\Phi(x)=\ln |x|\left(\int_{\R^2} \overline{\rho}(|y|)\,dy\right),\quad a(x,v)=\int_{|y|>|x|}\ln\left(\frac{|y|}{|x|}\right)\rho(|y|)\,dy\geq 0,
\end{equation*}the only difference with the setting of Proposition \ref{prop:initial} is that $\varphi$ can be possibly  unbounded on $\R$. Mimicking the proof of Proposition \ref{prop:initial} we still obtain
\begin{equation*}\begin{split}
\iint_{\R^2\times \R^2}|v|^k\,&(\overline{f}\,\1_{ \{\overline{f}\leq K\}})(x,v)\,h_0(x,v)\,dx\,dv\\
&\leq \|h_0\|_{L^\infty}K\int_{B(0,N)}\left(\int_{B(0,C(M+(\int \overline{\rho})|\ln |x||)^{1/2})}|v|^k\,dv\right)\,dx\\
&\leq C\int_{B(0,N)}\left(M+\big(\int \overline{\rho}\big)|\ln |x||\right)^{\frac{k+2}{2}}\,dx\\
&\leq (Ck)^{\frac{k}{2}},\end{split}\end{equation*} 
where we have used \eqref{ineq:verif} in the last inequality.

\end{proof}

\medskip

\noindent \textbf{Acknowledgments} {The author thanks Daniel Han-Kwan for interesting discussions. She is partially supported by the French ANR projects SchEq ANR-12-JS-0005-01 and GEODISP ANR-12-BS01-0015-01.}

\end{document}